\documentclass[letterpaper, 10 pt, conference]{ieeeconf}
\IEEEoverridecommandlockouts
\usepackage[T1]{fontenc}

\usepackage{amsmath,amssymb}
\usepackage{mathtools}
\usepackage{bm}
\usepackage{subfig}
\usepackage{algorithm}
\usepackage{algpseudocode}
\usepackage{tabularx}  

\newtheorem{definition}{Definition}

\newtheorem{proposition}{Proposition}
\newtheorem{theorem}{Theorem}

\newtheorem{remark}{Remark}
{}

\title{\LARGE \bf Computationally Efficient Density-Driven Optimal Control via Analytical KKT Reduction and Contractive MPC}
\author{Julian Martinez$^{1}$ and Kooktae Lee$^{1}$
\thanks{*This work was supported by NSF CAREER Grant CMMI-DCSD-2638508.}
\thanks{$^{1}$Julian Martinez and Kooktae Lee are with the Department of Mechanical and Aerospace Engineering, Texas Tech University, Lubbock, TX 79409, USA, email: mar85781@ttu.edu, kooktae.lee@ttu.edu.} 
}
\date{}

\begin{document}

\maketitle
\thispagestyle{empty}
\pagestyle{empty}

\begin{abstract}
Efficient coordination for collective spatial distribution is a fundamental challenge in multi-agent systems. Prior research on Density-Driven Optimal Control (D$^2$OC) established a framework to match agent trajectories to a desired spatial distribution. However, implementing this as a predictive controller requires solving a large-scale Karush-Kuhn-Tucker (KKT) system, whose computational complexity grows cubically with the prediction horizon. To resolve this, we propose an analytical structural reduction that transforms the $T$-horizon KKT system into a condensed quadratic program (QP). This formulation achieves $O(T)$ linear scalability, significantly reducing the online computational burden compared to conventional $O(T^3)$ approaches. Furthermore, to ensure rigorous convergence in dynamic environments, we incorporate a contractive Lyapunov constraint and prove the Input-to-State Stability (ISS) of the closed-loop system against reference propagation drift. Numerical simulations verify that the proposed method facilitates rapid density coverage with substantial computational speed-up, enabling long-horizon predictive control for multi-agent swarms.
\end{abstract}
\section{Introduction}

Multi-agent systems (MAS) have emerged as a transformative framework for executing complex tasks through the coordination of autonomous entities such as UAVs and mobile robots. A fundamental challenge in this coordination is the spatial distribution control problem. Traditionally viewed as a static coverage task, modern missions in precision agriculture \cite{seo2025tcst}, environmental monitoring \cite{notomista2022multi}, and surveillance \cite{du2017pursuing} now require agents to manage their collective presence as a continuous process. This necessitates a density matching approach where the swarm's spatial histogram tracks a desired distribution over time. Specifically, many scenarios demand that the swarm maintain a specific time-averaged presence over a given horizon to ensure thorough and continuous coverage of the task domain \cite{seo2025smcs}.

Translating macroscopic distribution goals into microscopic control laws is addressed through Eulerian and Lagrangian perspectives. The Eulerian approach treats the MAS as a continuous field governed by Partial Differential Equations (PDEs). Notable methods include mean-field feedback control \cite{zheng2021transporting} and backstepping-based formation control \cite{freudenthaler2020pde}. A related class of Eulerian techniques, such as Heat Equation Driven Area Coverage (HEDAC) \cite{hedac2017}, utilizes virtual potential fields to guide agents. While these frameworks provide rigorous analysis, they often focus on static or terminal matching. Even recent distributed online optimization techniques \cite{krishnan2025distributed} suffer from slow convergence or increased communication overhead as the predictive horizon extends.

Conversely, the Lagrangian approach focuses directly on discrete agent trajectories. While traditional approaches focus on static Voronoi-based coverage \cite{cortes2004} or stochastic vehicle routing for persistent service \cite{pavone2010adaptive}, and recent optimal transport-based frameworks \cite{lee2026optimal} primarily address terminal matching, significant progress has also been made in time-averaged behavior through ergodicity. A pioneering Lagrangian contribution is the spectral multiscale coverage framework \cite{mathew2011metrics}, which uses Fourier-based metrics to track target densities. Although robust, these spectral methods are primarily gradient-driven and often struggle to incorporate high-order dynamics or ensure formal optimality regarding control effort, especially over complex, multi-step predictive horizons.

Building upon these principles, the Density-Driven Optimal Control (D$^2$OC) framework provides a rigorous solution for optimal time-averaged distribution matching \cite{seo2025smcs, seo2025tcst, lee2025lcss}. However, existing D$^2$OC methodologies are hindered by $O(T^3)$ computational complexity and lack robust stability guarantees against reference drift. This paper overcomes these bottlenecks via a Contractive Lyapunov-based MPC framework with the following primary contributions: 
1) Analytical KKT Condensation: Unlike prior D$^2$OC works limited by $O(T^3)$ complexity, we derive an exact analytical reduction to $m$-dimensional QPs. This reduces the online cost to $O(T)$, enabling long-horizon implementation; 2) ISS under Reference Drift: We establish the first rigorous stability framework for D$^2$OC using a contractive Lyapunov constraint, proving the Input-to-State Stability (ISS) of the closed-loop system against reference propagation drift;
3) Real-time Decentralized Scalability: The proposed framework enables independent local optimization while maintaining global distribution matching. Simulations verify that the linear computational growth allows for high-fidelity control of swarms in real-time.

\section{Problem Description}
\noindent\textbf{Notations}: $\mathbb{R}^n$ and $\mathbb{R}^{n \times m}$ denote $n$-dimensional real vectors and $n \times m$ real matrices, respectively. $\mathbf{I}_n$ is the $n \times n$ identity matrix, and $\mathbf{1}_T$ denotes a $T$-dimensional vector of ones. The matrix transpose is $A^\top$, and $\otimes$ denotes the Kronecker product. $Q \succeq \mathbf{0}$ ($R \succ \mathbf{0}$) indicates that $Q$ is positive semi-definite ($R$ is positive definite). A discrete distribution is represented by a set of pairs $\{(q_j, \gamma_j)\}$, where $q_j$ and $\gamma_j$ are the position and weight of the $j$-th sample point, respectively.

\subsection{Multi-Agent System Modeling}
Consider a multi-agent system consisting of $N$ agents. Each agent $i \in \{1, \dots, N\}$ is governed by discrete-time Linear Time-Invariant (LTI) dynamics:
\begin{equation}
\mathsf{x}_i^{k+1} = A_i \mathsf{x}_i^k + B_i u_i^k, \quad y_i^k = C_i \mathsf{x}_i^k, \label{eqn: LTI system}
\end{equation}
where $k$ is a discrete-time index, $\mathsf{x}_i \in \mathbb{R}^n$ is the state, $u_i \in \mathbb{R}^m$ is the control input, and $y_i \in \mathbb{R}^d$ is the agent's position. It is important to note that the proposed D$^2$OC framework inherently supports fully distributed control, as the optimization for each agent is decoupled. While we consider heterogeneous dynamics ($A_i, B_i, C_i$) to emphasize the generality of our approach, the subscripts are omitted in the following sections for notational simplicity without loss of generality.

\subsection{Decentralized Density-Driven Optimal Control (D$^2$OC)}
The objective of D$^2$OC is to drive the empirical time-averaged distribution $\mu^k = \{(y_i^k, \alpha_i)\}_{i=1}^M$ to match a prescribed reference distribution $\nu = \{(q_j, \beta_j)\}_{j=1}^{N_{SP}}$. The symbols $y_i^k$ and $q_j$ denote the spatial positions, while the weights are uniformly defined as $\alpha_i = 1/M$ and $\beta_j = 1/N_{SP}$, representing the normalized contribution of each sample point to its respective distribution. Although the details are described in \cite{seo2025tcst,seo2025smcs}, the D$^2$OC framework is composed of a three-stage iterative cycle as follows:
\begin{itemize}[leftmargin=*,label=\textbullet]
    \item \textbf{Stage A: Local Sample Selection \& Control:} Agent $i$ identifies a local set of sample points (SPs), $\mathcal{S}^k$, near the current agent position $y_i^k$, and solves a $T$-horizon optimization problem to align its trajectory with the local density.
    \item \textbf{Stage B: Weight Adjustment:} After movement, agent $i$ reduces the weight of SPs at its current location by a certain mass amount to account for coverage progress.
    \item \textbf{Stage C: Information Exchange:} Agents share updated SPs' weights with neighbors within a communication range threshold to maintain a consistent view of the global coverage state.
\end{itemize}
Iterative execution of these stages ensures that the multi-agent system effectively sweeps high-priority regions encoded by the SPs.

This paper focuses on Stage A, where the matching between the agent's future trajectory and the local SPs is formulated via the Wasserstein distance, which serves as a metric to quantify the dissimilarity between two distributions \cite{villani2009optimal}. The optimization problem for agent $i$ at time $k$ is formulated with horizon $T$ by:
\begin{equation}
\begin{aligned}
    \min_{u_i} \quad & J = \Phi(\mathsf{x}_i^{k+T}) + \sum_{\ell=k}^{k+T-1} L(\mathsf{x}_i^{\ell}, u_i^{\ell}) \\
   \text{s.t.} \quad & \mathsf{x}_i^{k+1} = A\mathsf{x}_i^k + Bu_i^k,
\end{aligned} \label{eqn: minimization problem}
\end{equation}
where the stage cost $L(\cdot)$ and terminal cost $\Phi(\cdot)$ are defined as:
\vspace{-.15in}
\begin{align*}
    L(\mathsf{x}_i^{\ell}, u_i^{\ell}) &= \frac{1}{2}(\mathcal{W}_i^{\ell|k})^2 + \frac{1}{2}(\mathsf{x}_i^{\ell})^\top Q \mathsf{x}_i^{\ell} + \frac{1}{2}(u_i^{\ell})^\top R u_i^{\ell}, \\
    \Phi(\mathsf{x}_i^{k+T}) &= \frac{1}{2}(\mathcal{W}_i^{k+T|k})^2 + \frac{1}{2}(\mathsf{x}_i^{k+T})^\top Q \mathsf{x}_i^{k+T}
\end{align*}
with weighting matrices $Q \succeq \mathbf{0}$ and $R \succ \mathbf{0}$. The term
$(\mathcal{W}_i^{\ell|k})^2$, referred to as the local Wasserstein distance,
is defined as:
\begin{equation}
\small
    (\mathcal{W}_i^{\ell|k})^2 = \min_{\gamma_j \ge 0} \sum_{j \in \mathcal{S}^k} \gamma_j \|y_i^{\ell} - q_j\|^2,
    \; \text{s.t.} \sum_{j \in \mathcal{S}^k} \gamma_j = \alpha_i,\ \gamma_j \le \beta_j,
\end{equation}
where $\beta_j$ is the remaining mass at SP $j$; since the source is a
single point, the optimum is obtained by greedily filling $\gamma_j$ up
to $\beta_j$ in ascending order of distance until $\alpha_i$ is met.

\begin{theorem}[Analytical Solution \cite{seo2025smcs}]\label{Theorem: LTI D2OC Optimal control}
    For the LTI system \eqref{eqn: LTI system}, the optimal control sequence $\bar{u}_i^k:=[(u_i^k)^{\top},\cdots,(u_i^{k+T-1})^{\top}]^{\top}$ minimizing \eqref{eqn: minimization problem} is analytically given by:
    \begin{equation}
        \bar{u}_i^k = \mathcal{E}_{13}^{\top} F_{1} + \mathcal{E}_{23}^{\top} F_{2}, \label{eqn: optimal_u_2}
    \end{equation}
    where $\mathcal{E}_{ij}$ denotes the $(i,j)$-th block submatrix of $E^{-1}$ and the constituent blocks of the KKT matrix $E$ are defined as:
    {\allowdisplaybreaks
    \begin{align}
        &E_{11}=\bar{Q} \otimes \mathbf{I}_{T}, \ E_{23} = E_{32}^{\top} = B \otimes \mathbf{I}_{T}, \ E_{33} = R \otimes \mathbf{I}_{T}, \nonumber \\
        &E_{13}=E_{31}=E_{22}=\mathbf{0}, \ \bar{Q}=\Big\{\big(\sum_{j \in \mathcal{S}^{k}} \gamma_j\big)C^{\top}C+ Q\Big\}, \nonumber \\
        &E_{12} = E_{21}^{\top} = 
        \text{\footnotesize $
        \begin{bmatrix}
        -\mathbf{I}_n & A^{\top} & \mathbf{0} & \cdots & \mathbf{0} \\
        \mathbf{0} & -\mathbf{I}_n & A^{\top} & \ddots & \vdots \\
        \vdots & \ddots & \ddots & \ddots & \mathbf{0} \\
        \vdots & & \ddots & -\mathbf{I}_n & A^{\top} \\
        \mathbf{0} & \cdots & \cdots & \mathbf{0} & -\mathbf{I}_n
        \end{bmatrix} $}. \label{eqn: E and F}
    \end{align}}
    Moreover, $F_1$ and $F_2$ in \eqref{eqn: optimal_u_2} are defined by:
    \begin{equation*}
        \footnotesize
        F_1 = 
        \begin{bmatrix}
        (\sum_{j \in \mathcal{S}^{k}} \gamma_j)C^{\top}\overline{q}^{k}\\
        \vdots\\
        (\sum_{j \in \mathcal{S}^{k}} \gamma_j)C^{\top}\overline{q}^{k}
        \end{bmatrix}, \,
        F_2 = 
        \begin{bmatrix}
        -A\mathsf{x}_i^k \\
        \mathbf{0} \\
        \vdots \\
        \mathbf{0}
        \end{bmatrix},
    \end{equation*}
    where $\bar{q}^k = (\sum_{j \in \mathcal{S}^k} \gamma_j q_j) / (\sum_{j \in \mathcal{S}^k} \gamma_j)$ denotes the local barycenter of the sample points within $\mathcal{S}^k$.
\end{theorem}

\begin{remark}
    By utilizing the analytical block-inversion of the KKT matrix $E$, the D$^2$OC scheme avoids numerical instabilities associated with indefinite saddle-point problems, enabling real-time decentralized control. However, as the horizon $T$ increases, the dimension of the KKT system expands, naturally increasing the computational burden for direct matrix inversion.
\end{remark}
\section{Computational Efficiency Improvement via Multi-Step Reference Structural Reduction}

\subsection{From Fixed Target to Multi-Stage Prediction}
In our previous work \cite{seo2025smcs}, the optimal control utilized a stationary reference barycenter $\bar{q}^k$ fixed at time $k$. However, as an agent's movement progressively reduces nearby sample weights (Stage B), the local target barycenter systematically shifts even under a static global distribution. To anticipate this dynamic drift, we propose a multi-stage MPC framework incorporating a sequence of predicted barycenters $\{\bar{q}^{k+\ell+1}\}_{\ell=0}^{T-1}$, allowing the full $T$-horizon trajectory to align with the evolving local objectives inherent in the coverage process.

\begin{definition}[Iterative Sample Set and Weight Propagation]
For each prediction step $\ell \in \{0, \dots, T-1\}$, the local sample set $\mathcal{S}^{k+\ell}$ is collected around the predicted agent position. Consequently, the time-varying state penalty matrix $\bar{Q}^{k+\ell}$ and target barycenter $\bar{q}^{k+\ell+1}$ are recursively defined as:
\begin{equation}
\small
    \bar{Q}^{k+\ell} = \Big(\sum_{j \in \mathcal{S}^{k+\ell}} \gamma_j\Big) C^\top C + Q, \quad \bar{q}^{k+\ell+1} = \frac{\sum_{j \in \mathcal{S}^{k+\ell}} \gamma_j q_j}{\sum_{j \in \mathcal{S}^{k+\ell}} \gamma_j}.
\end{equation}
\end{definition}

\subsection{Analytical Condensation to $m$-Dimensional Quadratic Form}
Although incorporating multi-stage goals increases the theoretical complexity, we prove that the high-dimensional $T$-horizon problem originally formulated over the stacked input sequence $\bar{u}_i^k \in \mathbb{R}^{mT}$ in Theorem~\ref{Theorem: LTI D2OC Optimal control} can be exactly reduced to an efficient $m$-dimensional QP in terms of the actual control input $u_i^k \in \mathbb{R}^m$. By incorporating $T$-step preview information into the condensed objective function, the agent can strategically pursue coverage progress while maintaining a computational complexity that scales only linearly with the horizon $T$. This ensures that the online burden remains comparable to that of short-horizon controllers, enabling high-fidelity predictive control in real-time.

\begin{theorem}[Analytical Reduction of Time-Varying MPC]
\label{thm:full_reduction}
The optimal first-step control $u_i^k$ for the $T$-horizon problem \eqref{eqn: minimization problem} is the unique solution to the condensed QP:
\begin{equation}
\label{eq:condensed_qp_final}
\begin{aligned}
    \min_{u_i^k} \quad & \frac{1}{2} (u_i^k)^\top \mathbf{H}_T u_i^k + \mathbf{g}_T^\top u_i^k \quad
    \text{ s.t.} \quad  u_{\min} \le u_i^k \le u_{\max}
\end{aligned}
\end{equation}
where $\mathbf{H}_T \in \mathbb{R}^{m \times m}$ and $\mathbf{g}_T \in \mathbb{R}^m$ are analytically given by:
\vspace{-.2in}
\begin{align}
    \mathbf{H}_T &= R + \sum_{\ell=0}^{T-1} B^\top (A^\ell)^\top \bar{Q}^{k+\ell} A^\ell B, \\
    \mathbf{g}_T &= \sum_{\ell=0}^{T-1} B^\top (A^\ell)^\top \bar{Q}^{k+\ell} (A^{\ell+1}\mathsf{x}_i^k - \bar{q}^{k+\ell+1}).
\end{align}
\end{theorem}

\begin{proof}
The optimal control law in \eqref{eqn: optimal_u_2} is derived from the linear optimality conditions, which can be expressed as a KKT system of the form $Ez=F$, where $z = [\bar{\mathsf{x}}_i^\top, \bar{\lambda}_i^\top, \bar{u}_i^\top]^\top$ is the vector of decision variables consisting of the sequences of states, co-states (Lagrange multipliers), and control inputs over the prediction horizon $T$. The KKT matrix $E$ encapsulates the coupled dynamics and optimality constraints, while the right-hand side vector $F = [F_1^{\top}, F_2^{\top}, \mathbf{0}^{\top}]^{\top}$ incorporates the initial state and the time-indexed sequence of local barycenters $\{\bar{q}^{k+\ell+1}\}$.

By exploiting the specific block-sparse structure of $E$, one can obtain $\bar{\mathsf{x}}_i = E_{12}^{-\top}(F_2 - E_{23}\bar{u}_i^k)$, $\bar{\lambda}_i = E_{12}^{-1}(F_1 - E_{11}E_{12}^{-\top}(F_2 - E_{23}\bar{u}_i^k))$. Finally, substituting $\bar{\lambda}_i$ into the third row $E_{23}^\top \bar{\lambda}_i + E_{33}\bar{u}_i^k = \mathbf{0}$ yields the condensed system $\mathcal{H} \bar{u}_i^k = \mathcal{G}$, where:
\begin{align}
    \mathcal{H} &= E_{33} + E_{23}^\top E_{12}^{-1} E_{11} E_{12}^{-\top} E_{23}, \\
    \mathcal{G} &= E_{23}^\top E_{12}^{-1} (E_{11} E_{12}^{-\top} F_2 - F_1). \label{eq:schur_G}
\end{align}
Note that $E_{12}^{-1}$ is an upper block-triangular operator representing the adjoint state transition, which allows for the analytical reduction of the KKT system into a dense QP problem.

The control influence is captured by the block vector $\mathcal{K} = [ \mathcal{K}_0^\top, \dots, \mathcal{K}_{T-1}^\top ]^\top$, where each block is defined as $\mathcal{K}_\ell = -A^\ell B$.
The condensed Hessian $\mathbf{H}_T$, which is the $(1,1)$-block of $\mathcal{H}$, is then derived using $E_{11} = \text{diag}(\bar{Q}^{k}, \dots, \bar{Q}^{k+T-1})$:
\begin{equation}
\small
    \mathbf{H}_T = R + \sum_{\ell=0}^{T-1} \mathcal{K}_\ell^\top \bar{Q}^{k+\ell} \mathcal{K}_\ell = R + \sum_{\ell=0}^{T-1} B^\top (A^\ell)^\top \bar{Q}^{k+\ell} A^\ell B.
\end{equation}
For the linear term $\mathbf{g}_T$, we extract the first $m$-dimensional block of $\mathcal{G}$ defined in \eqref{eq:schur_G}. By observing the block-structure of $E_{23}^\top E_{12}^{-1}$, the first block-row is analytically given by $[-B^\top, -B^\top A^\top, \dots, -B^\top (A^{T-1})^\top]$. Substituting the initial condition $F_2 = [-A\mathsf{x}_i^k, \mathbf{0}, \dots]^\top$ and the target-dependent vector $F_1 = [ \bar{Q}^{k+1} \bar{q}^{k+1}, \dots, \bar{Q}^{k+T} \bar{q}^{k+T} ]^\top$, the first block $\mathbf{g}_T$ is evaluated as:
\begin{equation}
    \mathbf{g}_T = \sum_{\ell=0}^{T-1} (A^\ell B)^\top \bar{Q}^{k+\ell} ( A^{\ell+1} \mathsf{x}_i^k - \bar{q}^{k+\ell+1} )
\end{equation}
Note that for $\ell=0$, the term $B^\top \bar{Q}^k (A\mathsf{x}_i^k - \bar{q}^{k+1})$ correctly incorporates the initial state transition. This confirms the exact analytical reduction to an $m$-dimensional QP.
\end{proof}

\subsection{Computational Efficiency Analysis}

\begin{proposition}[Computational Efficiency of Reduced KKT]
\label{prop:computational_efficiency}
The structural reduction in Theorem~\ref{thm:full_reduction} achieves a complexity transformation from cubic to linear growth with respect to the horizon $T$ as follows:
\begin{enumerate}[leftmargin=*,label=(\alph*)]
    \item Full KKT Approach: Direct factorization of the $(2n+m)T$-dimensional system requires $\mathcal{O}(T^3 (2n+m)^3)$ operations, scaling cubically with the horizon.
    \item Proposed Reduced KKT: By utilizing the recursive construction of $\mathbf{H}_T$ and $\mathbf{g}_T$, the online cost is reduced to $\mathcal{O}(T n^2 m)$ for matrix updates and $\mathcal{O}(m^3)$ for solving the $m$-dimensional QP.
    \item Acceleration Factor: For $T \gg 1$, the computational speedup ratio relative to the full KKT approach is:
    \begin{equation}
        \mathcal{R} = \frac{\text{Cost}_{\rm KKT}}{\text{Cost}_{\rm Proposed}} \approx \frac{(2n+m)^3 T^3}{T n^2 m + m^3} \approx \mathcal{O}(T^2).
    \end{equation}
\end{enumerate}
\end{proposition}

\begin{proof}
The conventional KKT cost $\mathcal{C}_{\text{KKT}}$ depends on the dimension of the stacked vector $z \in \mathbb{R}^{(2n+m)T}$, leading to $\mathcal{C}_{\text{KKT}} \approx \mathcal{O}(((2n+m)T)^3)$. In the proposed method, we exploit the iterative relation $A^{\ell+1} = A \cdot A^\ell$ to compute the sequences $\{A^\ell B\}$ and $\{ (A^\ell)^\top \bar{Q} \}$. This recursive preparation allows $\mathbf{H}_T$ and $\mathbf{g}_T$ to be constructed in $T$ steps, each involving matrix-matrix products of order $\mathcal{O}(n^2 m)$. The total cost $\mathcal{C}_{\text{prop}}$ is then:
\begin{align}
    \mathcal{C}_{\text{prop}} &= \underbrace{\sum_{\ell=0}^{T-1} \mathcal{O}(n^2 m)}_{\text{Recursive Updates}} + \underbrace{\mathcal{O}(m^3)}_{\text{QP Solver}} 
    = \mathcal{O}(T \cdot n^2 m) + \mathcal{O}(m^3).
\end{align}
As $T$ increases, the cubic term $\mathcal{O}(T^3)$ in the KKT approach is dominated by the linear term $\mathcal{O}(T)$ in the proposed method, yielding an asymptotic acceleration of $\mathcal{O}(T^2)$. This ensures real-time feasibility for systems with long-horizon preview.
\end{proof}
\section{Contractive Lyapunov Constraint for Receding-Horizon D$^2$OC}
\label{sec:contractive_mpc}

In the previous section, we showed that the multi-goal D$^2$OC can be analytically reduced to an $m$-dimensional QP, achieving $O(T)$ computational complexity. However, since the reference barycenters $\{\bar{q}^{k+\ell}\}$ are dynamically updated via iterative sample selection and the spatial decay of density weights, the closed-loop stability cannot be directly inferred from standard infinite-horizon analysis. In this section, we introduce a contractive Lyapunov constraint designed as a Linear Matrix Inequality (LMI). This approach enforces solver-level stability at each MPC update, rigorously distinguishing the proposed multi-step scheme from a purely reactive horizon-1 QP.

\subsection{LMI-Based Stability Design via Schur Complement}

To analyze the tracking performance, we define the time-varying tracking error as $e^k := \mathsf{x}_i^k - \bar{q}^k$. Based on the discrete-time LTI dynamics \eqref{eqn: LTI system}, the error evolution is governed by:
\begin{equation}
e^{k+1} = A e^k + B u_i^k + d^k,
\label{eq:error_dynamics}
\end{equation}
where $d^k := A \bar{q}^k - \bar{q}^{k+1}$ represents the reference propagation drift. To guarantee the stability of the error dynamics, we employ a quadratic Lyapunov candidate $V(e^k) := (e^{k})^\top P e^k$ with a symmetric positive definite matrix $P \succ 0$. 

The core of our stability design is to enforce a strict contraction condition, $V(e^{k+1}) - V(e^k) \le - (e^{k})^\top Q_c e^k$, where $Q_c \succ 0$ determines the convergence rate. To incorporate this requirement into a convex optimization framework, we first formulate it as an LMI.

\begin{proposition}[Schur-Complement Contractive Constraint]
\label{prop:schur_contract}
The error dynamics \eqref{eq:error_dynamics} satisfies the contraction condition $V(e^{k+1}) - V(e^k) \le - (e^{k})^\top Q_c e^k$ if and only if the control input $u_i^k$ satisfies the following LMI:
\begin{equation}
\small
\begin{bmatrix}
(e^k)^\top(P - Q_c)e^k & (A e^k + B u_i^k + d^k)^\top P \\
P (A e^k + B u_i^k + d^k) & P
\end{bmatrix}
\succeq 0 .
\label{eq:schur_constraint}
\end{equation}
\end{proposition}

\begin{proof}
Consider a symmetric block matrix \\$\small M = \begin{bmatrix} X & Y^\top \\ Y & Z \end{bmatrix}$. According to the Schur complement lemma, $M \succeq 0$ if and only if $Z \succ 0$ and $X - Y^\top Z^{-1} Y \succeq 0$. By mapping the terms in \eqref{eq:schur_constraint} such that $X = (e^k)^\top(P - Q_c)e^k$, $Y = P(A e^k + B u_i^k + d^k)$, and $Z = P$, the condition $X - Y^\top Z^{-1} Y \succeq 0$ becomes:
$
(e^k)^\top(P - Q_c)e^k - (A e^k + B u_i^k + d^k)^\top P^\top P^{-1} P (A e^k + B u_i^k + d^k) \succeq 0.
$
Since $P$ is symmetric ($P^\top = P$), the term simplifies to $P P^{-1} P = P$. Substituting this and rearranging the inequality, we obtain:
\begin{equation*}
(A e^k + B u_i^k + d^k)^\top P (A e^k + B u_i^k + d^k) \le (e^{k})^\top (P - Q_c) e^k.
\end{equation*}
Noting that $e^{k+1} = A e^k + B u_i^k + d^k$ and $V(e^{k+1}) = (e^{k+1})^\top P e^{k+1}$, the above is equivalent to:
$V(e^{k+1}) \le (e^{k})^\top P e^k - (e^{k})^\top Q_c e^k,
$
which yields $V(e^{k+1}) - V(e^k) \le - (e^{k})^\top Q_c e^k$.
\end{proof}
\subsection{Numerical Reformulation for Real-Time Execution}

While Proposition \ref{prop:schur_contract} provides a robust theoretical foundation for stability via the LMI \eqref{eq:schur_constraint}, solving a Semidefinite Program (SDP) in real-time is computationally demanding for high-speed multi-agent coordination. To ensure numerical tractability, we leverage the Cholesky factorization of the Lyapunov matrix, $P = L^\top L$, to transform the matrix inequality into a Second-Order Cone (SOC) constraint.

By utilizing the properties of the Euclidean norm, the quadratic form $(e^{k+1})^\top P e^{k+1}$ is expressed as $\| L e^{k+1} \|_2^2$. Similarly, defining $S$ such that $S^\top S = P - Q_c$, with $\mathbf{0} \preceq Q_c \preceq P$ so that $S$ is well-defined regardless of whether $P$ and $Q_c$ commute, the upper bound $(e^{k})^\top (P - Q_c) e^k$ is exactly $\| S e^k \|_2^2$. Taking the square root of both sides, the contraction condition is reformulated as the following SOC constraint:
\begin{equation}
    \| L (A e^k + B u_i^k + d^k) \|_2 \le \| S \, e^k \|_2 .
    \label{eq:soc_derivation}
\end{equation}
This reformulation allows the stability-guaranteed controller to be implemented as a Quadratic Constrained Quadratic Program (QCQP), which is significantly faster to solve than general SDPs.

\subsection{QCQP Formulation}

Integrating the $T$-horizon condensed parameters $\mathbf{H}_T$ and $\mathbf{g}_T$ from Theorem \ref{thm:full_reduction}, the optimal control input $u_i^k$ is obtained by solving the following convex QCQP:
\begin{subequations}
\small
\label{eq:qcqp_formulation}
\begin{align}
\min_{u_i^k} & \quad \frac{1}{2} u_i^{k\top} \mathbf{H}_T u_i^k + \mathbf{g}_T^\top u_i^k \\
\text{s.t.} & \quad \| L (A e^k + B u_i^k + d^k) \|_2 \le \eta^k, \label{eq:socp_constraint} \\
& \quad u_\mathrm{min} \le u_i^k \le u_\mathrm{max},
\end{align}
\end{subequations}
where $\eta^k := \| S \, e^k \|_2$ is a constant scalar at each step $k$. The SOC constraint \eqref{eq:socp_constraint} ensures efficient numerical resolution while maintaining the recursive stability proven via the LMI framework.

\subsection{Closed-Loop Stability Analysis}

\begin{theorem}[Input-to-State Stability of Condensed D$^2$OC]
\label{thm:mpc_stability}
Assume the \emph{hard-constrained} QCQP \eqref{eq:qcqp_formulation} is feasible at each step $k$. Then, the tracking error $e^k$ is Input-to-State Stable (ISS) with respect to the reference drift $d^k$.
\end{theorem}

\begin{proof}
To prove ISS, we establish a dissipation inequality using the quadratic Lyapunov candidate $V(e^k) = (e^{k})^\top P e^k$. From the quadratic form properties, it holds that $\lambda_{\min}(P)\|e^k\|^2 \le V(e^k) \le \lambda_{\max}(P)\|e^k\|^2$. 

The feasibility of the contractive constraint \eqref{eq:socp_constraint} at each step $k$ ensures:
\begin{equation}
    V(e^{k+1}) \le V(e^k) - \lambda_{\min}(Q_c) \|e^k\|^2. \label{eq:v_decay_initial}
\end{equation}
By defining the decay ratio $\lambda = \lambda_{\min}(Q_c)/\lambda_{\max}(P) \in (0, 1)$, \eqref{eq:v_decay_initial} can be rewritten as:
\begin{equation}
\small
    V(e^{k+1}) \le (1 - \lambda) V(e^k). \label{eq:v_iterative}
\end{equation}
However, the actual error dynamics \eqref{eq:error_dynamics} involve the reference drift $d^k$. In the presence of $d^k$, the constraint \eqref{eq:socp_constraint} remains feasible only if the control $u_i^k$ can counteract the drift. Iterating \eqref{eq:v_iterative} from $j=0$ to $k-1$ yields:
\begin{equation}
\small
    V(e^k) \le (1-\lambda)^k V(e^0) + \sum_{j=0}^{k-1} (1-\lambda)^{k-j-1} \delta(d^j),
\end{equation}
where $\delta(\cdot)$ is a class $\mathcal{K}$ function representing the perturbation induced by the reference shift. Using the property of geometric series $\sum_{j=0}^{k-1} (1-\lambda)^{k-j-1} < 1/\lambda$, and taking the square root of the Lyapunov bounds, we obtain:
\begin{equation}
\footnotesize
\begin{split}
    \|e^k\| \le & \sqrt{\frac{\lambda_{\max}(P)}{\lambda_{\min}(P)}} (1-\lambda)^{k/2} \|e^0\| 
    + \frac{1}{\sqrt{\lambda_{\min}(P) \cdot \lambda}} \sup_{0 \le j < k} \|d^j\|.
\end{split}
\end{equation}
The first term corresponds to the class $\mathcal{KL}$ function $\beta(\|e^0\|, k)$ which decays exponentially, and the second term corresponds to the class $\mathcal{K}$ function $\gamma(\sup \|d^j\|)$, which defines the ultimate tracking bound. Thus, the system is ISS with respect to the reference propagation drift.
\end{proof}
\begin{remark}[Predictive Alignment and Feasibility]
The feasibility of \eqref{eq:qcqp_formulation} depends on whether the controller can manage the drift $\Vert{}d^k\Vert{}$ within its physical limits. By incorporating future barycenters, the linear term $\mathbf{g}_T$ allows the agent to preemptively align with future targets before the tracking error becomes critically large. This proactive movement keeps the state well within the safety boundary of the SOC constraint \eqref{eq:socp_constraint}, ensuring recursive feasibility even during rapid density shifts. 
\end{remark}


\subsection{Numerical Resolution via Soft-Constrained QCQP}
\label{subsec:numerical_qcqp}

To ensure recursive feasibility under strict control saturation constraints $u_i^k \in \mathcal{U}_i := \{u \mid u_{\min} \le u \le u_{\max}\}$, we reformulate the original QCQP \eqref{eq:qcqp_formulation} by introducing a slack variable $\epsilon \ge 0$. This relaxation transforms the hard contractive constraint \eqref{eq:socp_constraint} into a soft constraint, guaranteeing that a feasible solution exists even when the required energy decay rate is physically unattainable due to the agent's limited control authority. 

The relaxed optimization problem is formulated as follows:
\begin{subequations}
\label{eq:soft_qcqp_formulation}
\begin{align}
\min_{u_i^k, \epsilon} & \quad \frac{1}{2} (u_i^{k})^{\top} \mathbf{H}_T u_i^k + \mathbf{g}_T^\top u_i^k + \rho \epsilon^2 \\
\text{s.t.} & \quad \| L (A e^k + B u_i^k + d^k) \|_2 \le \mathcal{R}(e^k) + \epsilon, \label{eq:soft_soc_constraint} \\
& \quad u_i^k \in \mathcal{U}, \quad \epsilon \ge 0,
\end{align}
\end{subequations}
where $\mathcal{R}(e^k) := \| S \, e^k \|_2$ denotes the state-dependent stability radius and $\rho \gg 0$ is a penalty parameter. In this framework, the agent prioritizes the minimization of the tracking error while satisfying the stability-induced contraction to the greatest extent possible under saturation.
\begin{remark}[Scope of the ISS Guarantee]
Theorem \ref{thm:mpc_stability} applies to the hard-constrained controller \eqref{eq:qcqp_formulation}. For the soft-constrained implementation \eqref{eq:soft_qcqp_formulation}, the slack variable relaxes the contraction condition, and therefore the ISS guarantee of Theorem \ref{thm:mpc_stability} does not directly apply when $\epsilon^k > 0$.
\end{remark}

\subsection{Dual-Newton Solver for Real-Time Execution}\label{subsec:dual_newton}For resource-constrained on-board processing, solving \eqref{eq:soft_qcqp_formulation} via generic interior-point methods (IPMs) can be computationally expensive. Exploiting the fact that the control input dimension $m$ is typically much smaller than the state dimension, we derive a specialized Dual-Newton solver that exploits the analytical structure of the condensed QP.

By defining the stability residual as $f(u_i^k) := \| L(Ae^k + Bu_i^k + d^k) \|_2 - \mathcal{R}(e^k)$, the optimal control is obtained by iteratively updating the Lagrange multiplier $\mu \ge 0$ associated with \eqref{eq:soft_soc_constraint}. Setting the gradient of the Lagrangian to zero ($\nabla_u \mathcal{L} = 0$) yields the following control law parameterized by $\mu$:
$
u^*(\mu) = \text{proj}_{\mathcal{U}} \big( -(\mathbf{H}_T + \mu \mathbf{B}^\top P \mathbf{B})^{-1} \cdot (\mathbf{g}_T + \mu \mathbf{B}^\top P (A e^k + d^k)) \big).
$
The optimal $\mu^*$ is found via the Newton-Raphson update:
\begin{equation}
\footnotesize
\mu^{\nu+1} = \max \left( 0, \mu^{\nu} - \frac{f(u(\mu^{\nu}))}{\nabla_\mu f(u(\mu^{\nu}))} \right),
\end{equation}
where $\nu$ is the inner iteration index and $\nabla_\mu f$ denotes the sensitivity of the constraint residual with respect to the dual variable. This custom solver achieves sub-millisecond convergence, providing the necessary scalability for dense coverage tasks while maintaining the practical ISS property.

\begin{figure*}[!t]
    \centering
    \subfloat[Full KKT]{
    \includegraphics[width=0.275\linewidth]{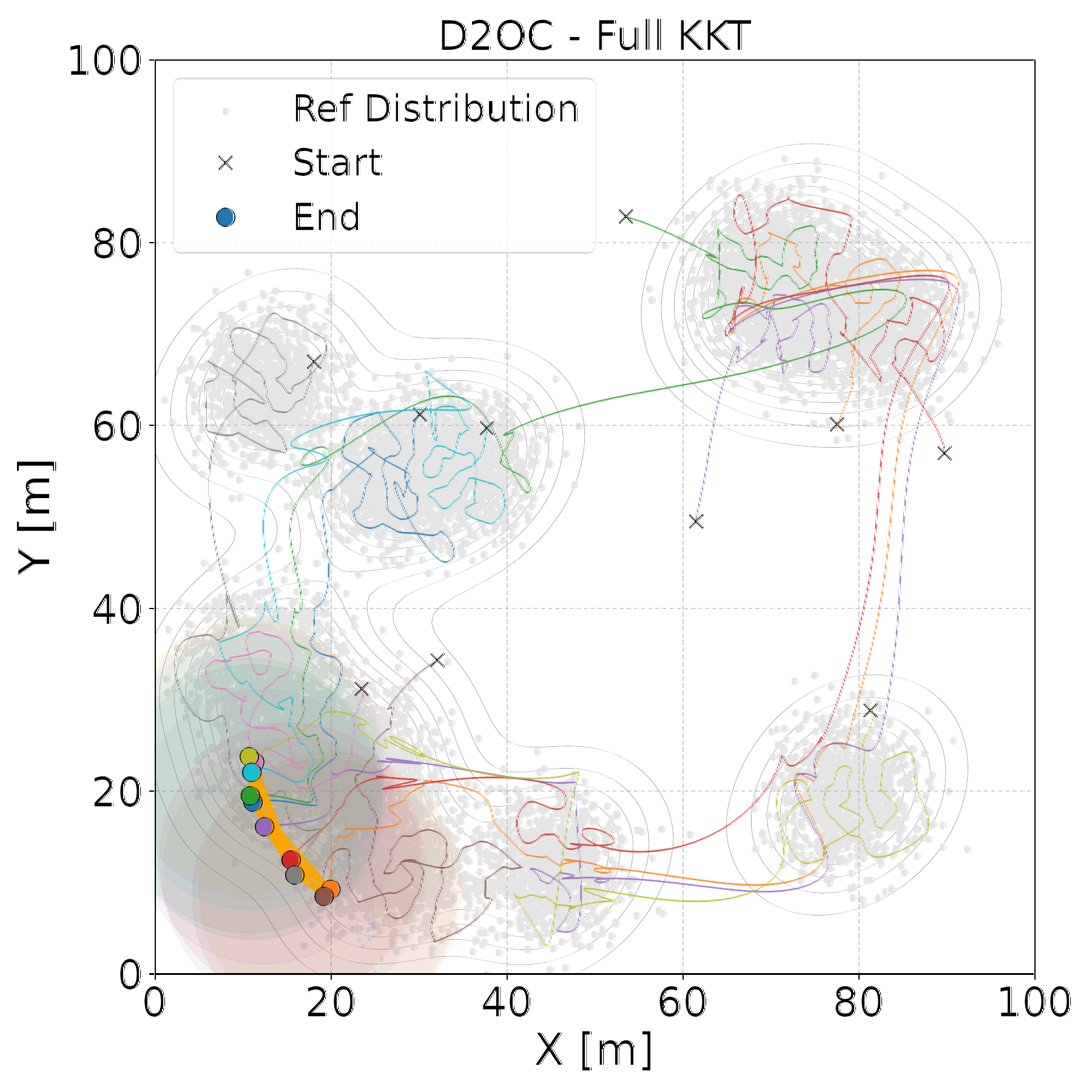}
    }\quad
    \subfloat[Reduced KKT]{
    \includegraphics[width=0.275\linewidth]{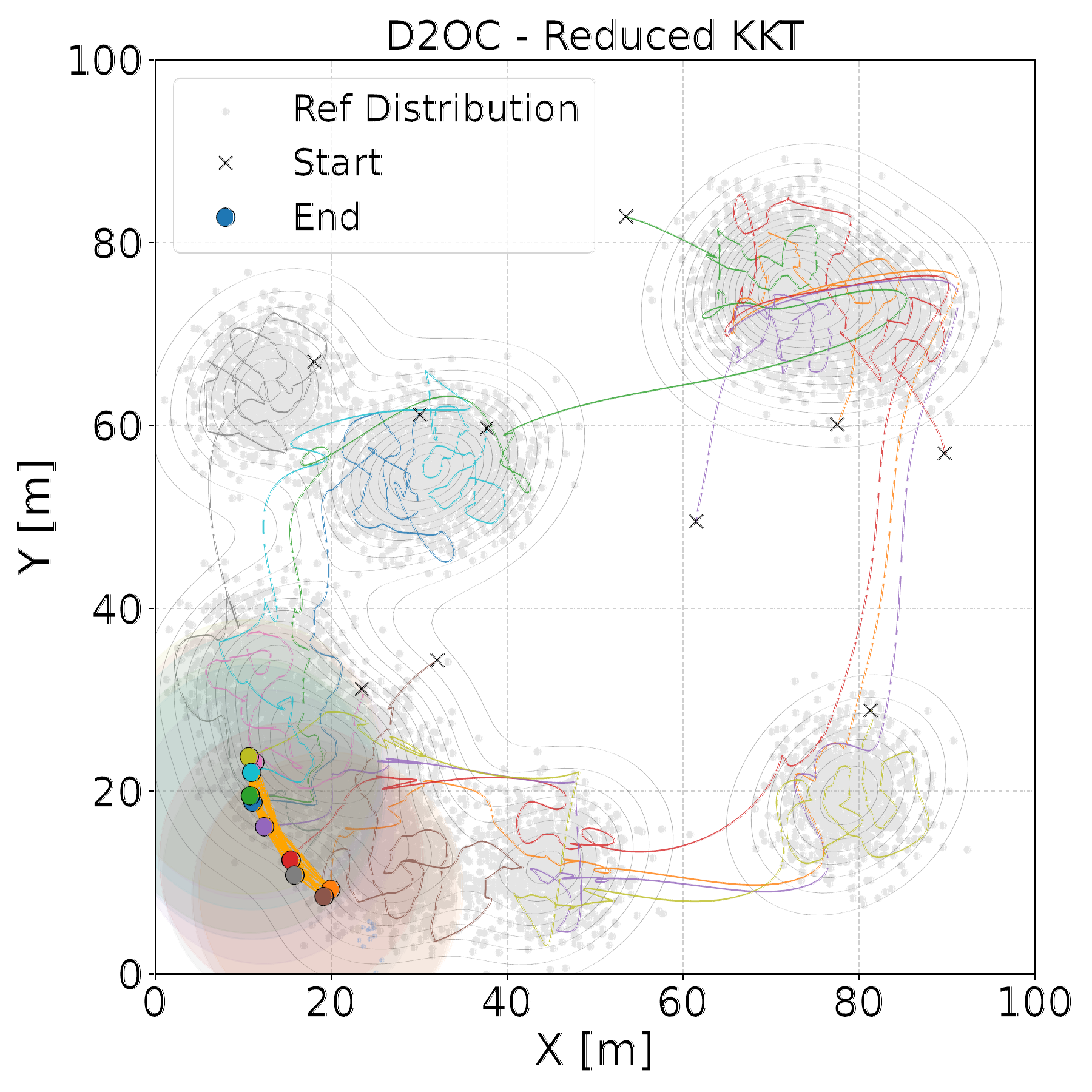}
    }\quad
    \subfloat[Reduced KKT w/ Stability]{
    \includegraphics[width=0.275\linewidth]{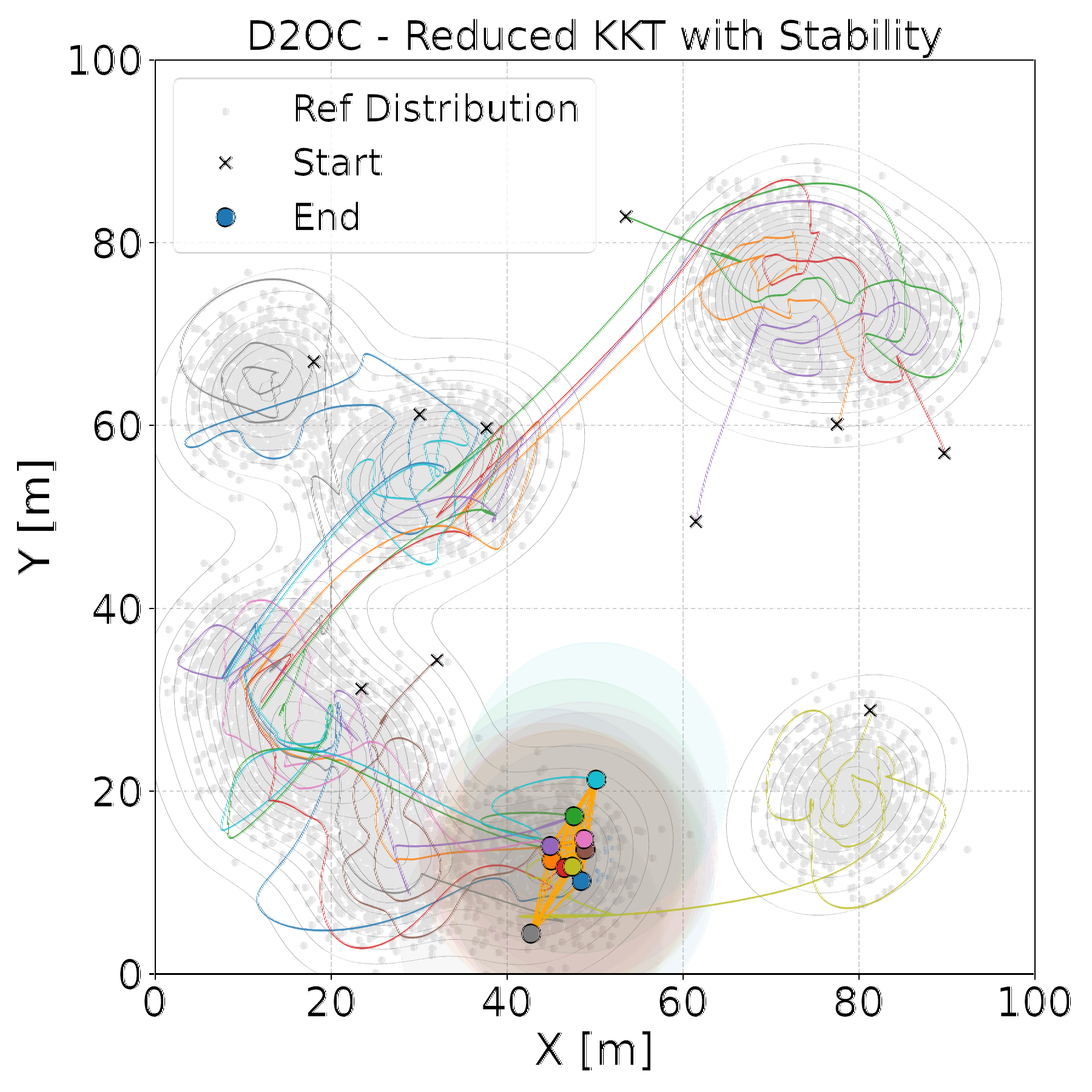}
    }
    \caption{Multi-agent coverage trajectories via D$^2$OC: (a) Full KKT baseline, (b) proposed reduced KKT formulation showing identical optimality, and (c) reduced KKT with recursive stability constraints.}
    \label{fig: traj}
\end{figure*}

\section{Simulation}
The performance and robustness of the proposed $O(T)$ D$^2$OC framework were evaluated against a conventional full KKT-based solver. We considered a multi-agent task with $N=10$ agents in a $100 \times 100\,\text{m}^2$ domain, where each agent follows linearized 8th-order quadrotor dynamics on a 2D plane. The communication range for each agent is set to $15\,\text{m}$, denoted by the semi-transparent circles around the agents in Fig.~\ref{fig: traj}. The target field is modeled as a non-convex Gaussian Mixture Model (GMM). To ensure consistent comparison, all simulations for horizons $T \in \{10, \dots, 60\}$ were executed until reaching 99\% global coverage.

\subsection{Trajectory Equivalency and Stability Analysis}
Multi-agent trajectories are presented in Fig.~\ref{fig: traj}, where full KKT solver and the proposed reduced KKT formulation results are shown in Figs.~\ref{fig: traj}(a) and (b), respectively. The paths are numerically identical, which confirms that the analytical reduction via Theorem~2 preserves the exact KKT optimality conditions of the original problem. Furthermore, Fig.~\ref{fig: traj}(c) demonstrates the results when augmented with stability constraints. Although paths deviate slightly to remain within the Lyapunov-based feasible region, agents successfully converge to the target distribution. This highlights the ability to ensure closed-loop stability without sacrificing the coverage mission.

\subsection{Computational Scalability and Efficiency}\label{subsec:comp_scalability}The computation time is depicted in Fig.~\ref{fig:comp_time_comparison}, where the horizon $T$ is varied from 10 to 60 in increments of 10. While the full KKT solver's execution time escalates to 26.46~ms at $T=60$, the proposed method maintains a consistent processing time of 2.70~ms. This 9.78-fold speedup validates the $O(T)$ complexity of our formulation, providing a substantial computational margin for high-frequency control loops where conventional solvers fail at large horizons.

Beyond mean execution time, the proposed method ensures timing predictability. As shown by the shaded regions in Fig.~\ref{fig:comp_time_comparison}, the full KKT solver suffers from numerical jitter with a standard deviation ($\sigma$) up to 3.77~ms. In contrast, the proposed method maintains extreme consistency with $\sigma = 0.25$~ms at $T=60$. This jitter suppression is a direct consequence of the reduced dimensionality and well-conditioned optimization landscape, which is paramount for safety-critical systems where timing-induced delays must be minimized.

\begin{figure}[!b]
    \centering
    \includegraphics[width=0.75\linewidth]{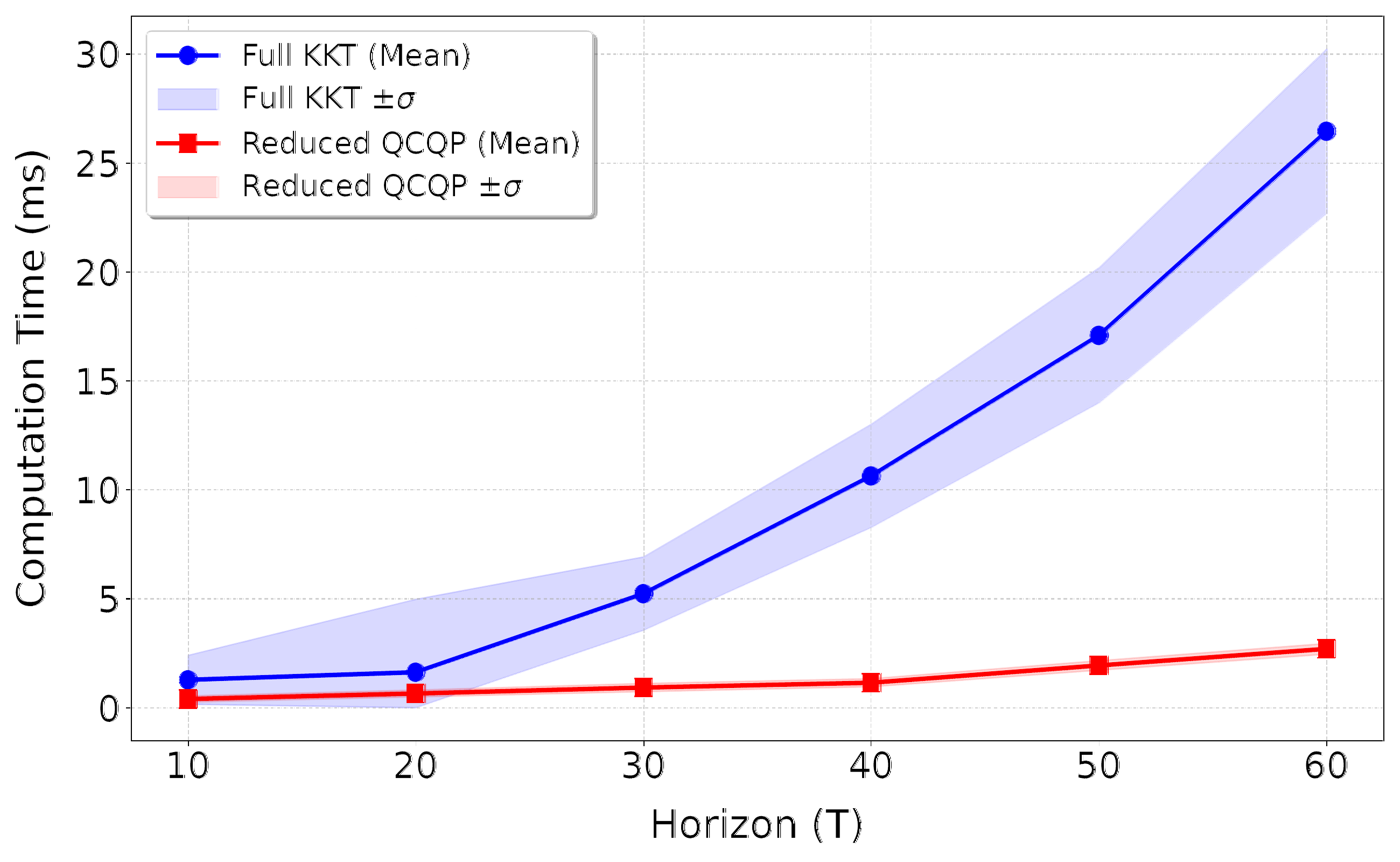}
    \caption{Computation time comparison b/w the full KKT and the reduced KKT with stability guarantee.}
    \label{fig:comp_time_comparison}
\end{figure}

\section{Conclusion}
In this paper, we presented an analytical structural reduction of the $T$-horizon D$^2$OC problem into a computationally efficient $m$-dimensional QP. By leveraging the specific block-triangular structure of the KKT matrix, we derived a closed-form solution that incorporates multi-stage target barycenters without the cubic growth in computational complexity. Furthermore, our stability analysis proves that the resulting condensed MPC framework, under the hard-constrained formulation, ensures ISS of the tracking error, even under time-varying reference drifts. The synergy between the linear term $\mathbf{g}_T$ and the contractive Lyapunov constraints ensures that the agent preemptively aligns with future density evolution while maintaining recursive feasibility. Our results demonstrate that the proposed framework achieves both the numerical stability of high-dimensional optimality conditions and the real-time scalability required for complex multi-agent coverage tasks.


\bibliographystyle{IEEEtran}
\bibliography{reference}

@book{villani2009optimal,
  title={Optimal transport: old and new},
  author={Villani, C{\'e}dric and others},
  volume={338},
  year={2009},
  publisher={Springer}
}

@ARTICLE{lee2026optimal,
  author={Lee, Kooktae},
  journal={IEEE Transactions on Automatic Control}, 
  title={Optimal Transport-Based Decentralized Multi-Agent Distribution Matching}, 
  year={2026},
  volume={},
  number={},
  pages={1-8},
  doi={10.1109/TAC.2026.3668445}}

@article{freudenthaler2020pde,
  title={PDE-based multi-agent formation control using flatness and backstepping: Analysis, design and robot experiments},
  author={Freudenthaler, Gerhard and Meurer, Thomas},
  journal={Automatica},
  volume={115},
  pages={108897},
  year={2020},
  publisher={Elsevier}
}

@article{zheng2021transporting,
  title={Transporting robotic swarms via mean-field feedback control},
  author={Zheng, Tongjia and Han, Qing and Lin, Hai},
  journal={IEEE Transactions on Automatic Control},
  volume={67},
  number={8},
  pages={4170--4177},
  year={2021},
  publisher={IEEE}
}

@article{krishnan2025distributed,
  title={Distributed online optimization for multi-agent optimal transport},
  author={Krishnan, Vishaal and Mart{\'\i}nez, Sonia},
  journal={Automatica},
  volume={171},
  pages={111880},
  year={2025},
  publisher={Elsevier}
}

@article{du2017pursuing,
  title={Pursuing an evader through cooperative relaying in multi-agent surveillance networks},
  author={Du, Sheng-Li and Sun, Xi-Ming and Cao, Ming and Wang, Wei},
  journal={Automatica},
  volume={83},
  pages={155--161},
  year={2017},
  publisher={Elsevier}
}

@inproceedings{notomista2022multi,
  title={Multi-robot persistent environmental monitoring based on constraint-driven execution of learned robot tasks},
  author={Notomista, Gennaro and Pacchierotti, Claudio and Giordano, Paolo Robuffo},
  booktitle={2022 international conference on robotics and automation (icra)},
  pages={6853--6859},
  year={2022},
  organization={IEEE}
}

@article{pavone2010adaptive,
  title={Adaptive and distributed algorithms for vehicle routing in a stochastic and dynamic environment},
  author={Pavone, Marco and Frazzoli, Emilio and Bullo, Francesco},
  journal={IEEE Transactions on automatic control},
  volume={56},
  number={6},
  pages={1259--1274},
  year={2010},
  publisher={IEEE}
}

@article{seo2025smcs,
  author={Seo, Sungjun and Lee, Kooktae},
  journal={IEEE Transactions on Systems, Man, and Cybernetics: Systems}, 
  title={Density-Driven Optimal Control for Efficient and Collaborative Multiagent Nonuniform Coverage}, 
  year={2025},
  volume={55},
  number={12},
  pages={9340-9354},
  doi={10.1109/TSMC.2025.3622075}}

@article{seo2025tcst,
  author={Seo, Sungjun and Lee, Kooktae},
  journal={IEEE Transactions on Control Systems Technology}, 
  title={Density-Driven Multidrone Coordination for Efficient Farm Coverage and Management in Smart Agriculture}, 
  year={2026},
  volume={34},
  number={2},
  pages={711-724},
  doi={10.1109/TCST.2025.3631091}}

@article{lee2025lcss,
  author={Lee, Kooktae and Brook, Ethan},
  journal={IEEE Control Systems Letters}, 
  title={Connectivity-Preserving Multi-Agent Area Coverage via Optimal-Transport-Based Density-Driven Optimal Control (D$^2$OC)}, 
  year={2025}
}

@article{hedac2017,
  title={Ergodicity-based cooperative multiagent area coverage via a potential field},
  author={Ivi{\'c}, Stefan and Crnkovi{\'c}, Bojan and Mezi{\'c}, Igor},
  journal={IEEE transactions on cybernetics},
  volume={47},
  number={8},
  pages={1983--1993},
  year={2016},
  publisher={IEEE}
}

@article{cortes2004,
  author={Cortes, J. and Martinez, S. and Karatas, T. and Bullo, F.},
  journal={IEEE Transactions on Robotics and Automation},
  title={Coverage control for mobile sensing networks},
  year={2004},
  volume={20},
  number={2},
  pages={243-255}
}

@article{mathew2011metrics,

title={Metrics for ergodicity and design of ergodic dynamics for multi-agent systems},

author={Mathew, George and Mezi{\'c}, Igor},

journal={Physica D: Nonlinear Phenomena},

volume={240},

number={4-5},

pages={432--442},

year={2011},

publisher={Elsevier}

}

\end{document}